\documentclass[a4paper]{amsart}
\usepackage{color}
\usepackage[latin1]{inputenc}
\usepackage[T1]{fontenc}
\usepackage{amsfonts}
\usepackage{amssymb}
\usepackage{amsmath}
\usepackage{amsthm}
\usepackage{stmaryrd}
\usepackage{eufrak}
\usepackage{graphicx,type1cm,eso-pic,color}
\usepackage{pstricks,pst-plot,pstricks-add}
\usepackage{float}
\newtheorem{theorem}{Theorem}[section]

\newtheorem{proposition}[theorem]{Proposition}
\newtheorem{corollary}[theorem]{Corollary}
\newtheorem{definition}[theorem]{Definition}
\newtheorem{assumption}[theorem]{Assumption}

\begin{document}
\setlength\arraycolsep{2pt}
\title{On a Fractional Stochastic Hodgkin-Huxley Model}
\author{Laure COUTIN*}
\author{Jean-Marc GUGLIELMI**}
\author{Nicolas MARIE***}
\address{*Institut de math\'ematiques de Toulouse, Toulouse, France}
\email{laure.coutin@math.univ-toulouse.fr}
\address{**American Hospital of Paris, Neuilly-sur-Seine, France}
\email{jean-marc.guglielmi@ahparis.org}
\address{***Laboratoire Modal'X, Universit\'e Paris 10, Nanterre, France}
\email{nmarie@parisnanterre.fr}
\address{***ESME Sudria, Paris, France}
\email{nicolas.marie@esme.fr}
\date{}
\maketitle
%

% Abstract.

%
\begin{abstract}
The model studied in this paper is a stochastic extension of the so-called neuron model introduced by Hodgkin and Huxley. In the sense of rough paths, the model is perturbed by a multiplicative noise driven by a fractional Brownian motion, with a vector field satisfying the viability condition of Coutin and Marie for $\mathbb R\times [0,1]^3$. An application to the modeling of the membrane potential of nerve fibers damaged by a neuropathy is provided.
\end{abstract}
\tableofcontents
\\
\\
\textbf{Acknowledgements.} Many thanks to Paul Raynaud de Fitte for its advices to improve the final version of this paper.
%

% Section : Introduction.

%
\section{Introduction}
The model studied in this paper is a stochastic extension of the so-called neuron model introduced by Hodgkin and Huxley in \cite{HH52}. The original model is a $4$-dimensional ordinary differential equation which models the dynamics of the ionic currents together with the membrane potential of the neuron. Precisely, the membrane potential of the neuron is modeled by
\begin{equation}\label{voltage_equation_intro}
C\dot V + I_{\textrm{Na}} + I_{\textrm K} + I_{\textrm L} = I,
\end{equation}
where $I_k := G_k(V - E_k)$ is the intensity of the ionic current $k$ (Na, K or L), $G_{\textrm L} :=\bar g_{\textrm L}$, $G_{\textrm K} :=\bar g_{\textrm K}n^4$ with
\begin{equation}\label{equation_n_intro}
\dot n =
\alpha_n(V)(1 - n) -\beta_n(V)n
\end{equation}
and $G_{\textrm{Na}} :=\bar g_{\textrm{Na}} m^3h$ with
\begin{equation}\label{equation_m_h_intro}
\left\{
\begin{array}{rcl}
 \dot m & = & \alpha_m(V)(1 - m) -\beta_m(V)m\\
 \dot h & = & \alpha_h(V)(1 - h) -\beta_h(V)h.
\end{array}
\right.
\end{equation}
All the parameters involving in the previous equations are defined at Section 2.
\\
\\
There are many deterministic extensions of Hodgkin-Huxley's model. For instance, in \cite{MR81}, Miller and Rinzel extended the Hodgkin-Huxley model in order to take into account that the propagation speed of an impulse is influenced by previous activity. In Lee et al. \cite{LNK98}, the authors studied a Hodgkin-Huxley model with no external signal. In \cite{NS14}, Nagy and Sweilam studied a deterministic fractional Hodgkin-Huxley model in which the derivatives are replaced by fractional derivatives.
\\
\\
In \cite{MS02}, Meunier and Segev proved that the behavior of $n$, $m$ and $h$ is partially random. In Saarinen et al. \cite{SLY08}, (\ref{equation_n_intro})-(\ref{equation_m_h_intro}) is perturbed by an additive Brownian noise. Unfortunately, in this case, the processes $n$, $m$ and $h$ are not $[0,1]$-valued as expected. In Cresson et al. \cite{CPS13}, in the sense of It\^o, (\ref{equation_n_intro})-(\ref{equation_m_h_intro}) is perturbed by a multiplicative noise driven by a Brownian motion with a vector field satisfying the viability condition of Aubin and DaPrato \cite{AD90} for $K :=\mathbb R\times [0,1]^3$.
\\
\\
In this paper, in the sense of rough paths, (\ref{equation_n_intro})-(\ref{equation_m_h_intro}) is perturbed by a multiplicative noise driven by a fractional Brownian motion with a vector field satisfying the viability condition of Coutin and Marie \cite{CM16} for $K$. A motivation for this extension of the Hodgkin-Huxley model is to control the regularity of the paths of $(V,n,m,h)$ via the Hurst parameter of the driving signal without losing the viability of $(V,n,m,h)$ in $K$. As suggested in Subsection 3.2, it should be interesting in applications because in some types of neuropathies there is a decrease over time of the regularity of the shape of the membrane potential of damaged nerve fibers (see Tasaki \cite{TASAKI56}).
\\
\\
In mathematical finance, the semimartingale property of the prices process is crucial in order to ensure the existence and the uniqueness of the risk-neutral probability measure. The It\^o stochastic calculus is then tailor-made to model prices in finance. This kind of condition isn't required in biological models. So, the pathwise stochastic calculus can be used to model dynamical systems in biology and the fractional stochastic extension of the Hodgkin-Huxley model studied in this paper is an example. For an application of the pathwise stochastic calculus in pharmacokinetics, see Marie \cite{MARIE15_PKPD}. As explained in Subsection 3.2, a motivation of the pathwise approach is to control the regularity of the paths of the model via the Hurst parameter of the driving signal.
\\
\\
Section 2 is a survey on the deterministic Hodgkin-Huxley neuron model and provides an appropriate formulation for the stochastic generalization introduced in Section 3. Section 3 deals with the existence, uniqueness and viability of the solution to the fractional stochastic Hodgkin-Huxley neuron model, but also with numerical simulations and an application to the modeling of the membrane potential of nerve fibers damaged by a neuropathy. Section 4 presents some perspectives and possible applications of the model. Finally, after a brief survey on the fractional Brownian motion and the pathwise stochastic calculus, the viability theorem used in this papier is proved in Appendix A.
%

% Section : The deterministic Hodgkin-Huxley model.

%
\section{The deterministic Hodgkin-Huxley model}
This section is a survey on the so called Hodgkin-Huxley neuron model (see Hodgkin and Huxley \cite{HH52}) and provides an appropriate formulation for the stochastic generalization introduced in Section 3.
%

% Subsection : The membrane potential.

%
\subsection{The membrane potential}
Let $V(t)$ be the displacement at time $t\in [0,T]$ of the membrane potential from its resting value. The signal $V$ satisfies
\begin{equation}\label{voltage_equation}
C\dot V(t) + I_{\textrm{ion}}(t) = I,
\end{equation}
where $C > 0$ is the membrane capacity per unit area, $I_{\textrm{ion}}(t)$ is the ionic current flowing across the membrane, in other words the ionic current density, and $I$ is the total membrane current density.
%

% Subsection : The ionic currents.

%
\subsection{The ionic currents}
In the Hodgkin-Huxley model, there are three ionic currents: Na (sodium ions), K (potassium ions) and L (other ions). It gives the following decomposition of $I_{\textrm{ion}}(t)$:
\begin{equation}\label{I_ion_decomposition}
I_{\textrm{ion}}(t) =
I_{\textrm{Na}}(t) +
I_{\textrm{K}}(t) +
I_{\textrm{L}}(t)
\end{equation}
with
\begin{displaymath}
I_k(t) := G_k(t)(V(t) - E_k),
\end{displaymath}
where $k$ is the current (Na, K or L) and $G_k(t)$ and $E_k$ are the conductance and the equilibrium potential for the $k$ ions respectively.
\\
\\
The potassium ions can only cross the membrane when four similar particles occupy a certain region of the membrane. It gives the following decomposition of $G_{\textrm K}(t)$:
\begin{displaymath}
G_{\textrm K}(t) =
\bar g_{\textrm K}n^4(t),
\end{displaymath}
where $\bar g_{\textrm K}$ is a normalization constant and $n(t)$ is the proportion of particles on the inside of the membrane. The signal $n$ satisfies
\begin{equation}\label{equation_n}
\dot n(t) =
\alpha_n(V(t))(1 - n(t)) -\beta_n(V(t))n(t)
\end{equation}
with
\begin{displaymath}
\alpha_n(v) :=
\frac{0.01\cdot (10 - v)}{\displaystyle{\exp\left(\frac{10 - v}{10}\right) - 1}}
\textrm{ and }
\beta_n(v) :=
0.125\cdot\exp\left(-\frac{v}{80}\right)
\end{displaymath}
for every $v\in\mathbb R$.
\\
\\
The sodium conductance is proportional to the number of sites on the inside of the membrane which are occupied simultaneously by three activating molecules but are not blocked by an inactivating molecule. It gives the following decomposition of $G_{\textrm{Na}}(t)$:
\begin{displaymath}
G_{\textrm{Na}}(t) =
\bar g_{\textrm{Na}}m^3(t)h(t),
\end{displaymath}
where $\bar g_{\textrm{Na}}$ is a normalization constant, $m(t)$ is the proportion of activating molecules on the inside of the membrane and $h(t)$ is the proportion of inactivating molecules on the outside of the membrane. The signal $m$ satisfies
\begin{equation}\label{equation_m}
\dot m(t) =
\alpha_m(V(t))(1 - m(t)) -\beta_m(V(t))m(t)
\end{equation}
with
\begin{displaymath}
\alpha_m(v) :=
\frac{0.1\cdot (25 - v)}{\displaystyle{\exp\left(\frac{25 - v}{10}\right) - 1}}
\textrm{ and }
\beta_m(v) :=
4\cdot\exp\left(-\frac{v}{18}\right)
\end{displaymath}
for every $v\in\mathbb R$. The signal $h$ satisfies
\begin{equation}\label{equation_h}
\dot h(t) =
\alpha_h(V(t))(1 - h(t)) -\beta_h(V(t))h(t)
\end{equation}
with
\begin{displaymath}
\alpha_h(v) :=
0.07\cdot\exp\left(-\frac{v}{20}\right)
\textrm{ and }
\beta_h(v) :=
\frac{1}{\displaystyle{\exp\left(\frac{30 - v}{10}\right) + 1}}
\end{displaymath}
for every $v\in\mathbb R$.
\\
\\
Note that the numerical values involved in $\alpha_n$, $\alpha_m$ and $\alpha_h$ come from Hodgkin and Huxley \cite{HH52}, Part II.
%

% Subsection : Existence, uniqueness and viability of the solution.

%
\subsection{Existence, uniqueness and viability of the solution}
It has been already proved, for instance in Aubin et al. \cite{ABS11}, Section 12.3.1, in the extended framework of the runs and impulse systems. Let's prove it via Corollary \ref{viability_corollary} for the sake of completeness.
\\
\\
By putting equations (\ref{equation_n}), (\ref{equation_m}) and (\ref{equation_h}) together, $P := (m,h,n)$ satisfies
\begin{equation}\label{gates_equation}
\dot P(t) = b_P(P(t),V(t)),
\end{equation}
where
\begin{displaymath}
b_P(p,v) :=
\begin{pmatrix}
 \alpha_m(v)(1 - p_1) -\beta_m(v)p_1\\
 \alpha_h(v)(1 - p_2) -\beta_h(v)p_2\\
 \alpha_n(v)(1 - p_3) -\beta_n(v)p_3
\end{pmatrix}
\end{displaymath}
for every $(p,v)\in [0,1]^3\times\mathbb R$.
\\
\\
By putting equations (\ref{voltage_equation}), (\ref{I_ion_decomposition}) and (\ref{gates_equation}) together, $X := (P,V)$ satisfies
\begin{equation}\label{HH_EDO}
\dot X(t) = b(X(t)),
\end{equation}
where for every $(p,v)\in [0,1]^3\times\mathbb R$,
\begin{displaymath}
b(p,v) :=
\begin{pmatrix}
 b_P(p,v)\\
 b_V(p,v)
\end{pmatrix}
\end{displaymath}
and
\begin{displaymath}
b_V(p,v) :=
\frac{1}{C}(I -
\bar g_{\textrm{Na}}\cdot
p_{1}^{3}\cdot p_2\cdot
(v - E_{\textrm{Na}}) -
\bar g_{\textrm K}\cdot
p_{3}^{4}\cdot
(v - E_{\textrm K}) -
\bar g_{\textrm L}\cdot
(v - E_{\textrm L})).
\end{displaymath}
The map $b$ fulfills assumptions \ref{viability_assumption} and \ref{existence_uniqueness_compact} with $\sigma\equiv 0$ and $K := [0,1]^3\times\mathbb R$. Therefore, by Corollary \ref{viability_corollary}, Equation (\ref{HH_EDO}) with $X_0\in K$ as initial condition has a unique solution $X$ defined on $[0,T]$ and viable in $K$. Note that it is crucial to ensure the viability of $P$ in $[0,1]^3$ since $m(t)$, $h(t)$ and $n(t)$ are proportions by definition.
%

% Section : A fractional generalization of the Hodgkin-Huxley model.

%
\section{A fractional generalization of the Hodgkin-Huxley model}
In this section, Equation (\ref{gates_equation}) which models the proportions $m(t)$, $h(t)$ and $n(t)$ will be perturbed by a multiplicative noise driven by a fractional Brownian motion, without loosing the viability of $P = (m,h,n)$ in $[0,1]^3$. In Subsection 3.1, the existence, uniqueness and viability of the solution $X$ to the fractional Hodgkin-Huxley model is proved by using the results of Appendix A. Subsection 3.2 deals with the control of the regularity of the paths of $X$ via the Hurst parameter of the driving fractional Brownian motion and an application to the modeling of the membrane potential of nerve fibers damaged by a neuropathy. Subsection 3.3 deals with some numerical simulations of $X$.
%

% Subsection : Existence, uniqueness and viability of the solution.

%
\subsection{Existence, uniqueness and viability of the solution}
Let $B$ be a fractional Brownian motion of Hurst parameter $H\in ]1/4,1[$ and consider also $\textrm B := (\textrm B_1,\textrm B_2,\textrm B_3)$, where $\textrm B_1$, $\textrm B_2$ and $\textrm B_3$ are three independent copies of $B$. In the sense of rough paths, consider the following stochastic extension of Equation (\ref{HH_EDO}):
\begin{equation}\label{fractional_HH}
dX(t) = b(X(t))dt +\sigma(X(t))d\textrm B(t),
\end{equation}
where $\sigma$ is a map from $\mathbb R^3$ into $\mathcal M_{4,3}(\mathbb R)$ such that $(b,\sigma)$ satisfies assumptions \ref{viability_assumption} and \ref{existence_uniqueness_compact} with $K = [0,1]^3\times\mathbb R$. For instance, with $\sigma_1,\sigma_2,\sigma_3 > 0$, one can put
\begin{displaymath}
\sigma(p,v) :=
\begin{pmatrix}
\sigma_1p_1(1 - p_1) & 0 & 0\\
0 & \sigma_2p_2(1 - p_2) & 0\\
0 & 0 & \sigma_3p_3(1 - p_3)\\
0 & 0 & 0
\end{pmatrix}
\end{displaymath}
for every $(p,v)\in [0,1]^3\times\mathbb R$.
\\
\\
Since the maps $b$ and $\sigma$ fulfill assumptions \ref{viability_assumption} and \ref{existence_uniqueness_compact} with $K$, by Corollary \ref{viability_corollary}, Equation (\ref{fractional_HH}) with $X_0\in K$ as initial condition has a unique solution $X$ defined on $[0,T]$ and viable in $K$.
\\
\\
Note that these ideas could be applied to extend other models. For instance, the Fitzhugh-Nagumo model (see Fitzhugh \cite{FITZHUGH55}).
%

% Subsection : Control of the solutions's paths regularity and applications.

%
\subsection{Control of the solution's paths regularity and applications}
By Proposition \ref{regularity_fbm}, for every $\alpha\in ]0,H[$, the paths of $\textrm B$ are $\alpha$-H\"older continuous. Moreover, by Theorem \ref{existence_uniqueness_rough}, Proposition \ref{regularity_solution_AN} and Proposition \ref{regularity_solution_DMN}, the solution of a rough differential equation inherits the H\"older regularity of its driving signal. So, the H\"older regularity of the paths of $P = (m,h,n)$, and then the regularity of the shape of the paths of $V$, are controlled by the Hurst parameter $H$ of $\textrm B$. Roughly speaking, the more $H$ is close to $1$, the more $P$ and $V$ have regular paths. Therefore, to take the fractional Brownian motion as driving signal in Equation (\ref{fractional_HH}) adds a way to control the regularity of the process $P$: the parameter $\sigma$ controls its global regularity and the parameter $H$ controls its local regularity.
\\
\\
Neurologists observed that in some types of neuropathies, there is a decrease over time of the regularity of the shape of the membrane potential of a damaged individual nerve fiber recorded several times during the disease (see Tasaki \cite{TASAKI56}). Assume that it is related to a perturbation of the dynamics of the ionic currents and let us provide a model to study the degeneracy of damaged nerve fibers over time.
\\
\\
Assume that the membrane potential of a damaged individual nerve fiber has been recorded $N\in\mathbb N^*$ times during the disease. According with the two facts previously stated in this subsection, for every $k\in\llbracket 1,N\rrbracket$, we suggest to model the $k$-th recording by Equation (\ref{fractional_HH}) with $H = H_k$, where $(H_1,\dots,H_N)$ is a vector of $]0,1[^N$ such that
\begin{displaymath}
H_k\geqslant H_{k + 1} > 1/4
\end{displaymath}
for every $k\in\llbracket 1,N - 1\rrbracket$.
%

% Subsection : Numerical simulations.

%
\subsection{Numerical simulations}
The purpose of this subsection is to provide some simulations of the Hodgkin-Huxley neuron model studied in this paper and to show why the viability condition on the vector field of Equation (\ref{fractional_HH}) is crucial.
\\
\\
Throughout this subsection, assume that $\textrm B$ is a fractional Brownian motion of Hurst parameter $H\in ]1/2,1[$. It is simulated via Wood-Chan's method (see Coeurjolly \cite{COEURJOLLY00}, Section 3.6). The solution to Equation (\ref{fractional_HH}) is approximated by the associated (explicit) Euler scheme (see Lejay \cite{LEJAY10}, Section 5).
\\
\\
The following values of the equilibrium potentials and of the normalized conductances come from Hodgkin and Huxley \cite{HH52}, Part II.
\begin{center}
\begin{tabular}{|l| l l |}
\hline
$k$ & $E_k$ (mV) & $\bar g_k$ (mS/cm$^2$)\\
\hline\hline
Na & 115 & 120\\
K & -12 & 36\\
L & 10.6 & 0.3\\
\hline
\end{tabular}
\end{center}
Put also $C := 1$ $\mu$F/cm$^2$ and $T := 50$ mS and consider the initial condition $X_0 := (V_0,m_0,h_0,n_0)$ with $V_0 := 0$ mV and
\begin{displaymath}
\begin{pmatrix}
m_0\\
h_0\\
n_0
\end{pmatrix}
:=
\begin{pmatrix}
\alpha_m(V_0)(\alpha_m(V_0) +\beta_m(V_0))^{-1}\\
\alpha_h(V_0)(\alpha_h(V_0) +\beta_h(V_0))^{-1}\\
\alpha_n(V_0)(\alpha_n(V_0) +\beta_n(V_0))^{-1}
\end{pmatrix}
\approx
\begin{pmatrix}
0.053\\
0.596\\
0.318
\end{pmatrix}.
\end{displaymath}
The deterministic Hodgkin-Huxley model (see Section 2) has Hopf bifurcations. The bifurcation parameter is the total membrane current density $I$. There exists $I_2 > I_1 > 0$ ($I_1\approx 3$ $\mu$A/cm$^2$ and $I_2\approx 6$ $\mu$A/cm$^2$) such that:
\begin{itemize}
 \item If $I\in [0,I_1]$, then $V$ returns at rest without spike.
 \item If $I\in ]I_1,I_2]$, then there is a single spike before $V$ returns at rest.
 \item If $I\in ]I_2,\infty[$, then there are multiple spikes. There is a limit cycle.
\end{itemize}
\noindent
On the following figure, in order to illustrate these behaviors, the Hodgkin-Huxley model is plotted for three different values of the bifurcation parameter $I$:
\begin{figure}[htbp]
\begin{center}
\includegraphics[scale = 0.3]{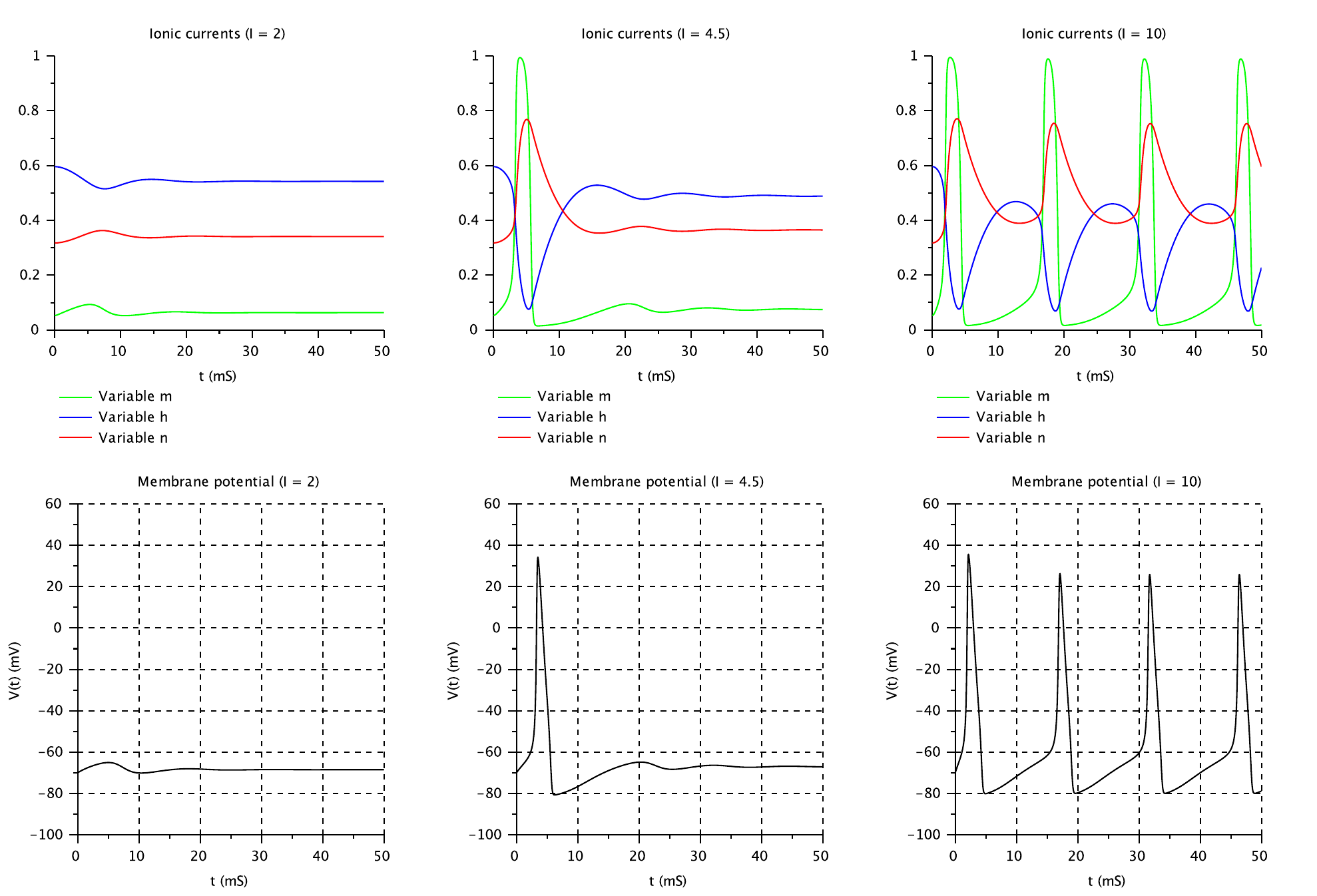}
\end{center}
\caption{Behaviors of the deterministic H-H model}
\label{fig:image1}
\end{figure}
\newline
Note that the stochastic Hodgkin-Huxley model studied in this paper (i.e. the solution $X$ to Equation (\ref{fractional_HH})) switches between these three different behaviors (see Figure 4).
\\
\\
In Equation (\ref{fractional_HH}), assume that:
\begin{displaymath}
\sigma(p,v) :=
0.25\cdot
\begin{pmatrix}
 p_1(1 - p_1) & 0 & 0\\
 0 & p_2(1 - p_2) & 0\\
 0 & 0 & p_3(1 - p_3)\\
 0 & 0 & 0
\end{pmatrix}.
\end{displaymath}
So, $(b,\sigma)$ satisfies assumptions \ref{viability_assumption} and \ref{existence_uniqueness_compact} with $K = [0,1]^3\times\mathbb R$. On the following figure, the solution to Equation (\ref{fractional_HH}) is plotted for $H = 0.55$ and $H = 0.95$:
\begin{figure}[htbp]
\begin{center}
\includegraphics[scale = 0.3]{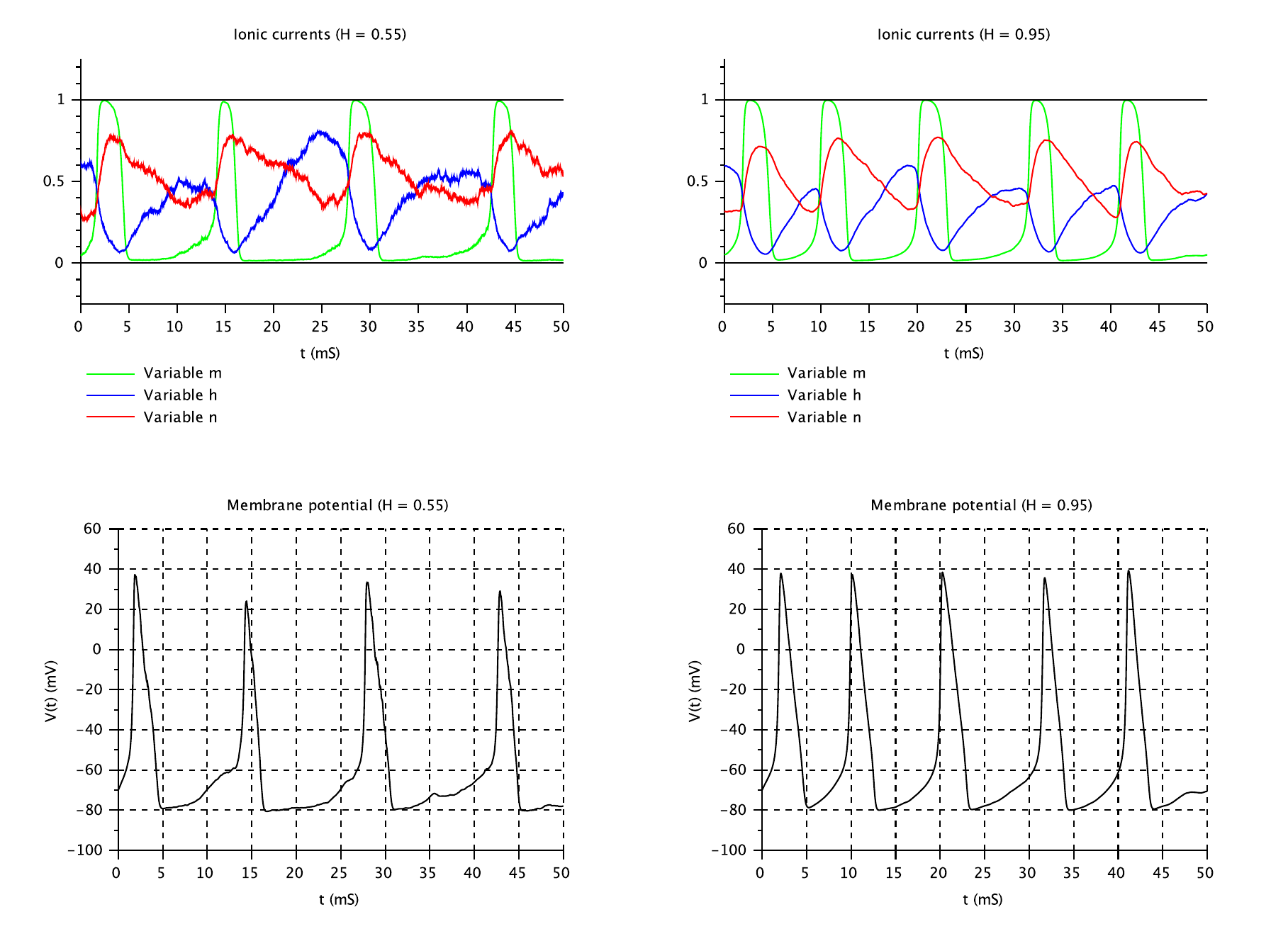}
\end{center}
\caption{Stochastic H-H model with viability condition}
\label{fig:image1}
\end{figure}
\newline
One can see that $X$ is viable in $K$ as mentioned in Subsection 3.1 and $H$ controls the local regularity of the paths of $P = (m,h,n)$ as mentioned in Subsection 3.2. Via $P$, the value of $H$ impacts also the regularity of the shape of the paths of the process $V$.
\\
\\
Now, in order to show that Assumption \ref{viability_assumption} with $K$ is crucial, let us simulate Equation (\ref{fractional_HH}) with an additive noise ($\sigma\equiv 0.25$). Then, $X$ is not viable in $K$ and the model is not appropriate:
\begin{figure}[htbp]
\begin{center}
\includegraphics[scale = 0.3]{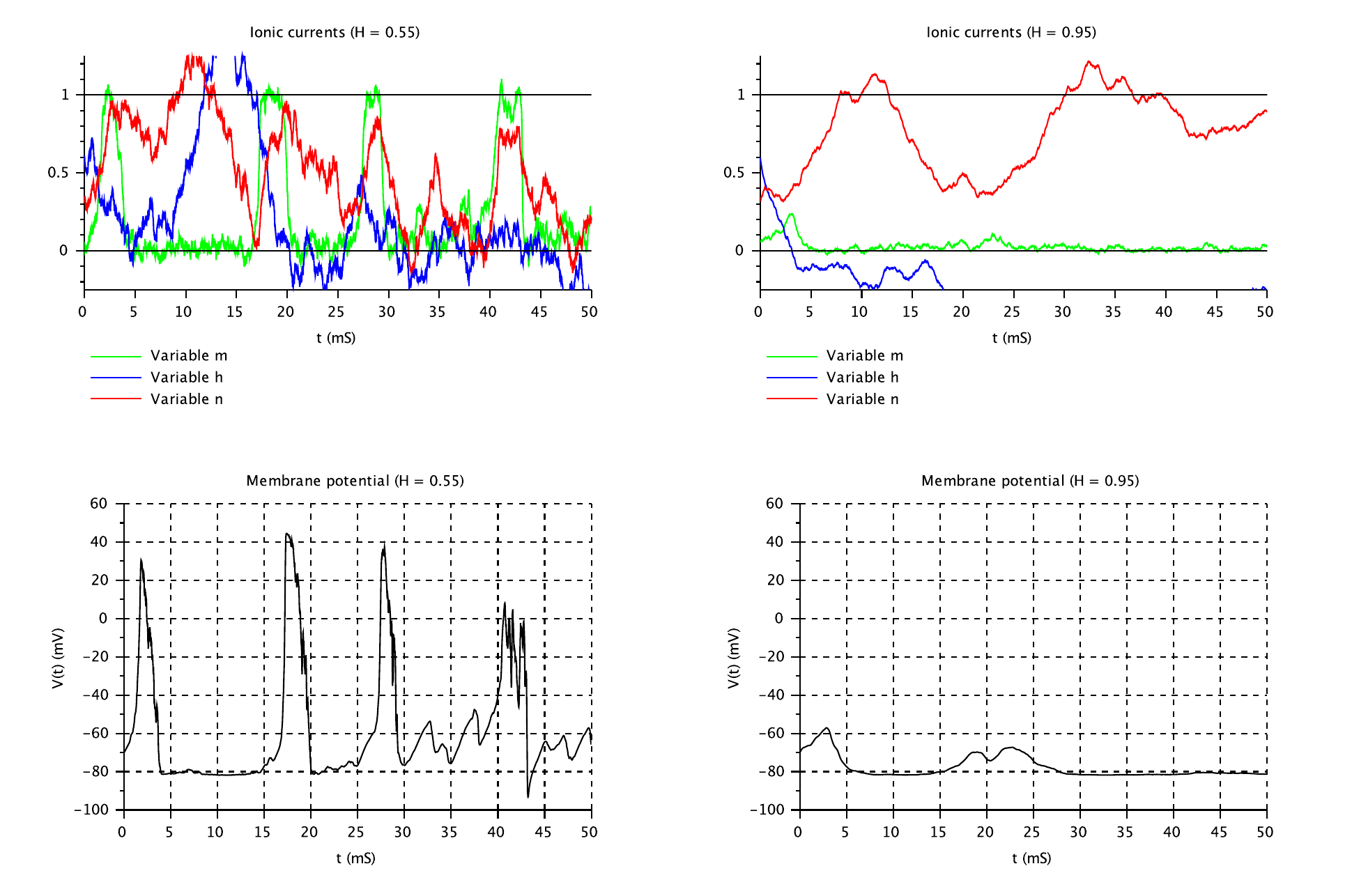}
\end{center}
\caption{Stochastic H-H model with additive noise}
\label{fig:image1}
\end{figure}
%
%

% Section : Discussion and perspectives.

%
\section{Discussion and perspectives}
The stochastic neuron model studied in this paper is an extension of the deterministic Hodgkin-Huxley model obtained by perturbing the dynamics of the ionic currents by a multiplicative fractional noise. By the viability theorem proved in Appendix A, the functions $m$, $h$ and $n$ are still $[0,1]$-valued. Thanks to the rough differential equations framework, to take the fractional Brownian motion as driving signal allows to control the regularity of the paths of $X$. The model can be simulated easily and we are now investigating some applications of our model to the modeling of the potential of an individual nerve fiber during neuropathies.
\\
\\
On the figure below, for $T := 1000$ mS, $I := 10$ $\mu$A/cm$^2$, $H := 0.9$ and $\sigma_k := 0.25$ for every $k\in\llbracket 1,3\rrbracket$, the stochastic model switches between the three behaviors mentioned at Subsection 3.3:
\begin{figure}[htbp]
\begin{center}
\includegraphics[scale = 0.3]{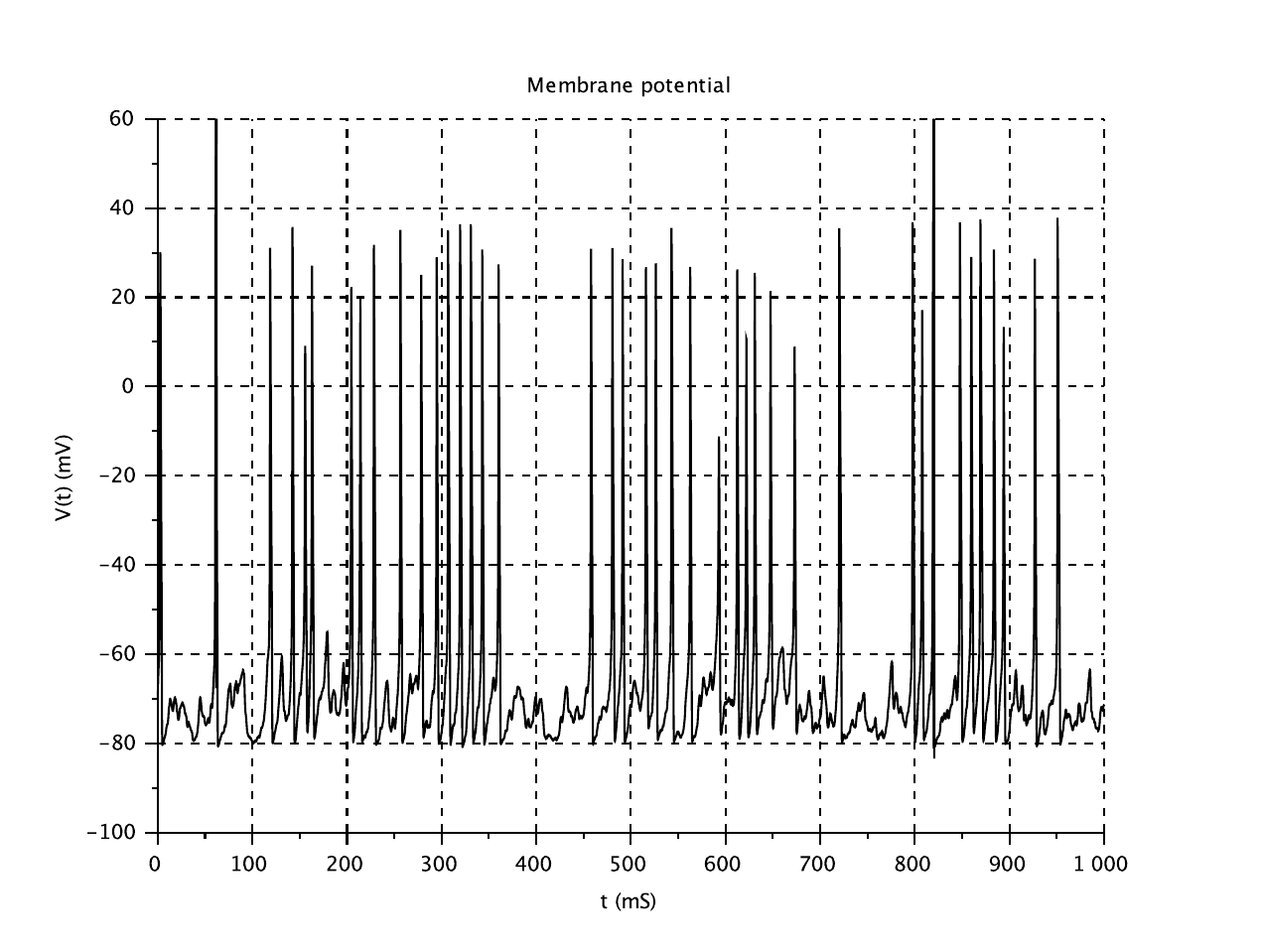}
\end{center}
\caption{Behaviors switching}
\label{fig:image1}
\end{figure}
\newline
An interesting research perspective is to study equilibrium stability and bifurcations of the fractional Hodgkin-Huxley model, for a random current $I$, in the random dynamical systems framework (see Arnold \cite{ARNOLD97}, Chapter 9).
\appendix
%

% Section : A viability theorem for differential equations driven by a fractional Brownian motion.

%
\section{A viability theorem for differential equations driven by a fractional Brownian motion}
The first subsection deals with the regularity of the paths of the fractional Brownian motion and differential equations driven by a fractional Brownian motion. The second subsection deals with a viability result which is crucial to study the fractional Hodgkin-Huxley model provided in this paper.
\\
\\
\textbf{Notations.} Consider $d,e\in\mathbb N^*$.
\begin{enumerate}
 \item The euclidean scalar product (resp. norm) on $\mathbb R^d$ is denoted by $\langle .,.\rangle$ (resp. $\|.\|$). For every $x\in\mathbb R^d$, its $j$-th coordinate with respect to the canonical basis of $\mathbb R^d$ is denoted by $x_j$ for every $j\in\llbracket 1,d\rrbracket$.
 \item The space of the matrices of size $d\times e$ is denoted by $\mathcal M_{d,e}(\mathbb R)$. For every $M\in\mathcal M_{d,e}(\mathbb R)$, its $(i,j)$-th coordinate with respect to the canonical basis of $\mathcal M_{d,e}(\mathbb R)$ is denoted by $M_{i,j}$ for every $(i,j)\in\llbracket 1,d\rrbracket\times \llbracket 1,e\rrbracket$.
 \item The space of the continuous functions from $[0,T]$ into $\mathbb R^d$ is denoted by $C^0([0,T],\mathbb R^d)$ and equipped with the uniform norm $\|.\|_{\infty,T}$ such that
 \begin{displaymath}
 \|f\|_{\infty,T} :=
 \sup_{t\in [0,T]}
 \|f(t)\|
 \end{displaymath}
 for every $f\in C^0([0,T],\mathbb R^d)$.
 \item The space of the $\alpha$-H\"older continuous maps from $[s,t]$ into $\mathbb R^d$ with $\alpha\in ]0,1[$ and $s,t\in [0,T]$ such that $s < t$ is denoted by $C^{\alpha}([s,t],\mathbb R^d)$:
 \begin{displaymath}
 C^{\alpha}([s,t],\mathbb R^d) :=
 \left\{
 f : [s,t]\rightarrow\mathbb R^d :
 \sup_{s\leqslant u < v\leqslant t}
 \frac{\|f(v) - f(u)\|}{|v - u|^{\alpha}} <\infty
 \right\}.
 \end{displaymath}
 Note that for every $\alpha,\beta\in ]0,1[$ such that $\alpha\leqslant\beta$,
 \begin{displaymath}
 C^{\beta}([s,t],\mathbb R^d)
 \subset
 C^{\alpha}([s,t],\mathbb R^d).
 \end{displaymath}
 Let $\|.\|_{\alpha,s,t}$ be the semi-norm on $C^{\alpha}([s,t],\mathbb R^d)$ defined by:
 \begin{displaymath}
 \|f\|_{\alpha,s,t} :=
 \sup_{s\leqslant u < v\leqslant t}
 \frac{\|f(v) - f(u)\|}{|v - u|^{\alpha}}
 \textrm{ $;$ }
 \forall f\in C^{\alpha}([s,t],\mathbb R^d).
 \end{displaymath}
 \item The space of the $N\in\mathbb N^*$ times continuously differentiable maps from $\mathbb R^d$ into $\mathbb R^e$ is denoted by $C^N(\mathbb R^d,\mathbb R^e)$.
\end{enumerate}
%

% Subsection : Differential equations driven by a fBm.

%
\subsection{Differential equations driven by a fractional Brownian motion}
This subsection deals with basics on differential equations driven by a fractional Brownian motion.
%

% Definition : Fractional Brownian motion.

%
\begin{definition}\label{fbm}
Let $B$ be a centered Gaussian process. It is a fractional Brownian motion if and only if there exists $H\in ]0,1[$, called Hurst parameter of $B$, such that
\begin{displaymath}
\normalfont{\textrm{cov}}(B(s),B(t)) =
\frac{1}{2}(|s|^{2H} +
|t|^{2H} -
|t - s|^{2H})
\end{displaymath}
for every $(s,t)\in [0,T]^2$.
\end{definition}
%

% Proposition : Regularity of the paths of the fBm.

%
\begin{proposition}\label{regularity_fbm}
Let $B$ be a fractional Brownian motion of Hurst parameter $H\in ]0,1[$. The paths of $B$ are $\alpha$-H\"older continuous for every $\alpha\in ]0,H[$.
\end{proposition}
\noindent
See Nualart \cite{NUALART06}, Section 5.1.
\\
\\
Let $B$ be a fractional Brownian motion of Hurst parameter $H\in ]1/4,1[$ and consider $\textrm B := (\textrm B_1,\dots,\textrm B_e)$, where $\textrm B_1,\dots,\textrm B_e$ are $e\in\mathbb N^*$ independent copies of $B$. Consider also $(\textrm B^N)_{N\in\mathbb N^*}$, a sequence of piecewise linear approximations of $\textrm B$.
\\
\\
In the sequel, $(\Omega,\mathcal A,\mathbb P)$ is the canonical probability space for $B$.
\\
\\
Consider the differential equation
\begin{equation}\label{main_equation_appendix}
X(t) =
X_0 +\int_{0}^{t}b(X(s))ds +
\int_{0}^{t}\sigma(X(s))d\textrm B(s),
\end{equation}
where $X_0\in\mathbb R^d$ and $b$ (resp. $\sigma$) is a Lipschitz continuous map from $\mathbb R^d$ into $\mathbb R^d$ (resp. $\mathcal M_{d,e}(\mathbb R)$).
%

% Definition : Solutions of a differential equation in the sense of rough paths.

%
\begin{definition}\label{solution_differential_equation}
In the sense of rough paths, a process $X := (X(t))_{t\in [0,T]}$ is a solution on $[0,T]$ to Equation (\ref{main_equation_appendix}) if and only if
\begin{displaymath}
\lim_{N\rightarrow\infty}
\|X^N - X\|_{\infty,T} = 0,
\end{displaymath}
where for every $N\in\mathbb N^*$, $X^N$ is the solution on $[0,T]$ of the ordinary differential equation
\begin{displaymath}
X^N(t) =
X_0 +\int_{0}^{t}b(X^N(s))ds +
\int_{0}^{t}\sigma(X^N(s))d\normalfont\textrm B^N(s).
\end{displaymath}
\end{definition}
\noindent
In the sequel, the maps $b$ and $\sigma$ satisfy the following assumption.
%

% Assumption : Assumption on $\mu$ and $\sigma$ (existence and uniqueness).

%
\begin{assumption}\label{existence_uniqueness_assumption}
$b\in C^{[1/H] + 1}(\mathbb R^d,\mathbb R^d)$ and $\sigma\in C^{[1/H] + 1}(\mathbb R^d,\mathcal M_{d,e}(\mathbb R))$, their derivatives are bounded and $b$ (resp. $\sigma$) is Lipschitz continuous from $\mathbb R^d$ into itself (resp. $\mathcal M_{d,e}(\mathbb R)$).
\end{assumption}
%

% Theorem : Existence and uniqueness of the solution.

%
\begin{theorem}\label{existence_uniqueness_rough}
Under Assumption \ref{existence_uniqueness_assumption}, Equation (\ref{main_equation_appendix}) with $X_0\in\mathbb R^d$ as initial condition has a unique solution denoted by $\pi_{b,\sigma}(0,X_0,\normalfont\textrm B)$ and its paths belong to $C^{\alpha}([0,T],\mathbb R^d)$ for every $\alpha\in ]0,H[$.
\end{theorem}
\noindent
See Friz and Victoir \cite{FV10}, Theorem 10.26, Exercice 10.55 and Exercice 10.56.
\\
\\
In some cases, at least locally, the paths of the solution to Equation (\ref{main_equation_appendix}) are $\alpha$-H\"older continuous for every $\alpha\in ]0,H[$, but not $H$-H\"older continuous. In other words, the solution to Equation (\ref{main_equation_appendix}) inherits the H\"older regularity of $\textrm B$. The two following results apply to the stochastic extensions of the Hodgkin-Huxley model simulated in Subsection 3.3. The proofs of these results are similar to the proof of Proposition 4.10 in the 3rd unpublished arXiv version of Castaing, Marie and Raynaud de Fitte \cite{CMR17}.
%

% Proposition : Regularity of the solution (additive noise).

%
\begin{proposition}\label{regularity_solution_AN}
Under Assumption \ref{existence_uniqueness_assumption}, if $\sigma$ is constant, then the paths of the solution to Equation (\ref{main_equation_appendix}) are $\alpha$-H\"older continuous on $[s,t]$ for every $\alpha\in ]0,H[$, but not $H$-H\"older continuous.
\end{proposition}
%

% Proof.

%
\begin{proof}
Consider $\omega\in\Omega$ and assume that there exists $(s,t)\in [0,T]^2$ such that $s < t$ and $X(\omega)$ is $H$-H\"older continuous on $[s,t]$. Since the map
\begin{displaymath}
u\in [s,t]\longmapsto\int_{s}^{u}b(X(\omega,r))dr
\end{displaymath}
is Lipschitz continuous, it is $H$-H\"older continuous. Moreover, for every $(u,v)\in [s,t]^2$ such that $u < v$,
\begin{displaymath}
\textrm B(\omega,v) -\textrm B(\omega,u) =
\frac{1}{\sigma}\left(X(\omega,v) - X(\omega,u) -\int_{u}^{v}b(X(\omega,r))dr\right).
\end{displaymath}
So, $\textrm B(\omega)$ should be $H$-H\"older continuous on $[s,t]$ as linear combination of $H$-H\"older continuous functions on $[s,t]$, but this is wrong. So, necessarily, $X(\omega)$ is not $H$-H\"older continuous on $[s,t]$.
\end{proof}
%

% Proposition : Regularity of the solution (diagonal multiplicative noise).

%
\begin{proposition}\label{regularity_solution_DMN}
Consider $H\in ]1/2,1[$, $X_0\in\mathbb R^d$ and $\omega\in\Omega$. Assume that $d = e$, $(b,\sigma)$ fulfills Assumption \ref{existence_uniqueness_assumption} and $\sigma_{k,l}\equiv 0$ for every $(k,l)\in\llbracket 1,d\rrbracket^2$ such that $k\not= l$. For every $(s,t)\in [0,T]^2$ such that $s < t$ and
\begin{equation}\label{regularity_solution_DMN_1}
(\sigma_{k,k}\circ
\pi_{b,\sigma}(0,X_0,\normalfont\textrm B(\omega)))([s,t])\subset
\mathbb R^*
\textrm{ $;$ }
\forall k\in\llbracket 1,d\rrbracket,
\end{equation}
the map $\pi_{b,\sigma}(0,X_0,\normalfont\textrm B(\omega))$ is $\alpha$-H\"older continuous on $[s,t]$ for every $\alpha\in ]0,H[$, but not $H$-H\"older continuous.
\end{proposition}
%

% Proof.

%
\begin{proof}
Consider $(s,t)\in [0,T]^2$ such that $s < t$ and (\ref{regularity_solution_DMN_1}) is true. Let $k\in\llbracket 1,d\rrbracket$ be arbitrarily chosen and for every $(u,v)\in [s,t]^2$ such that $v < u$, consider
\begin{displaymath}
I_k(u,v)(\omega) :=
\int_{u}^{v}\sigma_{k,k}(X(\omega,r))d\textrm B_k(\omega,r),
\end{displaymath}
where $X(\omega) :=\pi_{b,\sigma}(0,X_0,\textrm B(\omega))$. Since $\sigma$ (resp. $X(\omega)$) is continuous on $\mathbb R$ (resp. $[s,t]$), by (\ref{regularity_solution_DMN_1}), there exists $\sigma_{k}^{*}(\omega) > 0$ such that:
\begin{equation}\label{regularity_solution_DMN_2}
|\sigma_{k,k}(X(\omega,u))|\geqslant\sigma_{k}^{*}(\omega) > 0
\textrm{ $;$ }
\forall u\in [s,t].
\end{equation}
Assume that the map $u\in [s,t]\mapsto I_k(s,u)(\omega)$ is $H$-H\"older continuous on $[s,t]$. Consider $\alpha\in ]0,H[$. By Young-Love's estimate (see Friz and Victoir \cite{FV10}, Theorem 6.8), there exists a deterministic constant $c > 0$ such that for $(u,v)\in [s,t]^2$ satisfying $u < v$,
\begin{eqnarray*}
 |I_k(u,v)(\omega) -\sigma_{k,k}(X(\omega,u))(\textrm B_k(\omega,v) -\textrm B_k(\omega,u))|
 & \leqslant & c|v - u|^{2\alpha}\\
 & &
 \times\|X(\omega)\|_{\alpha,s,t}\|\textrm B_k(\omega)\|_{\alpha,s,t}.
\end{eqnarray*}
So, by Inequality (\ref{regularity_solution_DMN_2}):
\begin{eqnarray*}
 |\textrm B_k(\omega,v) -\textrm B_k(\omega,u)|
 & \leqslant &
 \frac{1}{\sigma_{k}^{*}(\omega)}
 |v - u|^H\\
 & &
 \times
 (T^{2\alpha - H}\|X(\omega)\|_{\alpha,s,t}\|\textrm B_k(\omega)\|_{\alpha,s,t} +
 \|I_k(s,.)(\omega)\|_{H,s,t}).
\end{eqnarray*}
Since $\textrm B_k(\omega)$ is not $H$-H\"older continuous on $[s,t]$, there is a contradiction. Therefore, $u\in [s,t]\mapsto I_k(s,u)(\omega)$ is not $H$-H\"older continuous on $[s,t]$. In conclusion, $X_k(\omega)$ is not $H$-H\"older continuous on $[s,t]$ by Equation (\ref{main_equation_appendix}).
\end{proof}
%

% Subsection : The viability theorem.

%
\subsection{The viability theorem}
This subsection deals with a corollary of the viability theorem proved in Coutin and Marie \cite{CM16} which is crucial to study the fractional Hodgkin-Huxley model provided in this paper.
\\
\\
Let $K\subset\mathbb R^d$ be a closed convex set.
%

% Definition : Viability.

%
\begin{definition}\label{viability}
A function $\varphi : [0,T]\rightarrow\mathbb R^d$ is viable in $K$ if and only if
\begin{displaymath}
\varphi(t)\in K
\textrm{ $;$ }
\forall t\in [0,T].
\end{displaymath}
\end{definition}
%

% Definition : Invariant sets.

%
\begin{definition}\label{invariant_sets}
Under Assumption \ref{existence_uniqueness_assumption}, the subset $K$ is invariant for $\pi_{b,\sigma}(0,.;\normalfont\textrm B)$ if and only if, for any initial condition $x_0\in K$, the paths of $\pi_{b,\sigma}(0,x_0;\normalfont\textrm B)$ are viable in $K$.
\end{definition}
\noindent
\textbf{Notation.} For every $x\in K$, the normal cone to $K$ at $x$ is denoted by $N_K(x)$:
\begin{displaymath}
N_K(x) :=
\{s\in\mathbb R^d :
\forall y\in K
\textrm{$,$ }
\langle s,y - x\rangle\leqslant 0\}.
\end{displaymath}
It the sequel, the maps $b$ and $\sigma$ satisfy the following assumption.
\begin{assumption}\label{viability_assumption}
For every $x\in\partial K$ and $s\in N_K(x)$,
\begin{displaymath}
\langle s,b(x)\rangle\leqslant 0
\end{displaymath}
and
\begin{displaymath}
\langle s,\sigma_{.,k}(x)\rangle = 0
\textrm{ $;$ }
\forall k\in\llbracket 1,e\rrbracket.
\end{displaymath}
\end{assumption}
%

% Proposition : Viability theorem.

%
\begin{proposition}\label{viability_theorem}
Under Assumption \ref{existence_uniqueness_assumption}, $K$ is invariant for $\pi_{b,\sigma}(0,.;\normalfont\textrm B)$ if and only if $b$ and $\sigma$ satisfy Assumption \ref{viability_assumption}.
\end{proposition}
\noindent
See Coutin and Marie \cite{CM16}, Proposition 5.3.
\\
\\
Finally, let's prove that Assumption \ref{existence_uniqueness_assumption} can be relaxed when $K := C\times\mathbb R$ and $C\subset\mathbb R^{d - 1}$ is a compact and convex set.
%

% Assumption : Existence and uniqueness of the solution with viability in a convex compact set.

%
\begin{assumption}\label{existence_uniqueness_compact}
$b\in C^{[1/H] + 1}(\mathbb R^d,\mathbb R^d)$, $\sigma\in C^{[1/H] + 1}(\mathbb R^d,\mathcal M_{d,e}(\mathbb R))$ with $\sigma_{d,.}\equiv 0$ and $b_d$ is Lipschitz continuous from $\mathbb R^d$ into $\mathbb R$.
\end{assumption}
%

% Corollary : Corollary of the viability theorem.

%
\begin{corollary}\label{viability_corollary}
Let $C\subset\mathbb R^{d - 1}$ be a convex and compact set and consider $K := C\times\mathbb R$. If $b$ and $\sigma$ satisfy assumptions \ref{viability_assumption} and \ref{existence_uniqueness_compact}, then Equation (\ref{main_equation_appendix}) with $X_0\in K$ as initial condition has a unique solution $X$ defined on $[0,T]$ and viable in $K$.
\end{corollary}
%

% Proof.

%
\begin{proof}
Since $b\in C^{[1/H] + 1}(\mathbb R^d,\mathbb R^d)$ and $\sigma\in C^{[1/H] + 1}(\mathbb R^d,\mathcal M_{d,e}(\mathbb R))$, there exists $\tau\in ]0,T]$ such that Equation (\ref{main_equation_appendix}) with $X_0\in K$ as initial condition has a unique solution $X$ on $[0,\tau[$.
\\
\\
Since $b$ and $\sigma$ satisfy Assumption \ref{viability_assumption}, by Proposition \ref{viability_theorem} applied to Equation (\ref{main_equation_appendix}) on $[0,\tau[$:
\begin{displaymath}
X(t)\in K
\textrm{ $;$ }
\forall t\in [0,\tau[.
\end{displaymath}
So, $\widetilde X := (X_1,\dots,X_{d - 1})$ is bounded on $[0,\tau[$ by a constant $M > 0$ because $C$ is a bounded subset of $\mathbb R^{d - 1}$.
\\
\\
Moreover, since $b_d$ is Lipschitz continuous from $\mathbb R^d$ into $\mathbb R$, there exists a constant $c_1 > 0$ such that
\begin{displaymath}
|b_d(x)|\leqslant c_1(1 +\|x\|)
\textrm{ $;$ }
\forall x\in\mathbb R^d.
\end{displaymath}
So, for every $t\in [0,\tau[$,
\begin{eqnarray*}
 |X_d(t)| & \leqslant &
 |X_d(0)| +\int_{0}^{t}|b_d(X(s))|ds\\
 & \leqslant &
 |X_d(0)| + c_1\int_{0}^{t}(1 +\|\widetilde X(s)\| + |X_d(s)|)ds\\
 & \leqslant &
 |X_d(0)| + c_1T(1 + M) + c_1\int_{0}^{t}|X_d(s)|ds.
\end{eqnarray*}
Then, by Gronwall's lemma,
\begin{displaymath}
|X_d(t)|
\leqslant
c_2e^{c_1T}
\end{displaymath}
with
\begin{displaymath}
c_2 :=
|X_d(0)| + c_1T(1 + M).
\end{displaymath}
Therefore, $X$ doesn't explode as $t\rightarrow\tau$.
\\
\\
In conclusion, by Friz and Victoir \cite{FV10}, Theorem 10.21, $X$ is defined on $[0,T]$ and by Proposition \ref{viability_theorem}, it is viable in $K$.
\end{proof}
%

% References.

%

%
\end{document}